\documentclass[final]{elsarticle}

\journal{Journal of Statistical Planning and Inference}

\usepackage{natbib}     
\usepackage{amsthm}     
\usepackage{amsmath}    
\usepackage{amsfonts}   
\usepackage{amssymb}    
\usepackage{mathrsfs}   
\usepackage{parskip}    
\usepackage{calc}       
\usepackage{enumitem}   
\usepackage{graphicx}   
\usepackage{lpic}       
\usepackage[english]{babel}
\addto\captionsenglish{}
\RequirePackage{hyperref}
\hypersetup{
    colorlinks=true, 
    linktoc=all,     
    linkcolor=blue,  
    citecolor=blue,
    urlcolor = blue,
    }
\usepackage{cleveref}
\crefformat{equation}{(#2#1#3)}

\newcommand{\ie}{{\it i.e.}}
\newcommand{\cf}{{\it cf.}}
\newcommand{\eg}{{\it e.g.}}

\newcommand{\scrB}{{\mathscr B}}

\newcommand{\scrD}{{\mathscr D}}

\newcommand{\scrF}{{\mathscr F}}
\newcommand{\scrG}{{\mathscr G}}

\newcommand{\scrP}{{\mathscr P}}

\newcommand{\scrX}{{\mathscr X}}

\newcommand{\al}{{\alpha}}
\newcommand{\ep}{{\epsilon}}

\newcommand{\tht}{{\theta}}

\newcommand{\Tht}{{\Theta}}
\renewcommand{\emptyset}{{\varnothing}}

\newcommand{\set}[1]{\left\{ #1 \right\}}
\newcommand{\ctg}{\mathbin{\lhd}}

\newcommand{\ft}[2]{{\textstyle{\frac{#1}{#2}}}}

\newcommand{\conv}[1]%
  {{\mathrel{\,\xrightarrow{\widthof{\,#1\,}}\,}}}
\newcommand{\convas}[1]%
  {{\mathrel{\,\xrightarrow{\widthof{\,#1\text{-a.s.}\,}}\,}}}
\newcommand{\convprob}[1]%
  {{\mathrel{\,\xrightarrow{\widthof{\,#1\,}}\,}}}
\newcommand{\convweak}[1]%
  {{\mathrel{\,\xrightarrow{\widthof{\,#1\text{-w.}\,}}\,}}}

\newcommand{\twobytwo}[4]%
  {\left(\begin{array}{cc} #1 & #2 \\ #3 & #4 \end{array}\right)}
\newcommand{\twovec}[2]%
  {\left({\begin{array}{c} #1\\#2 \end{array}}\right)}
\newcommand{\contig}{\mathop{\vartriangleleft}}

\newcommand{\ceil}[1]{\left\lceil #1 \right\rceil}
\newcommand{\floor}[1]{\left\lfloor #1 \right\rfloor}
\newcommand{\diam}{{\mathrm{diam}}}

\newtheoremstyle{customtheorem}
  {0.5em}
  {0.2em}
  {\itshape}
  {}
  {\scshape}
  {}
  {1ex}
  {}

\theoremstyle{customtheorem}
\newtheorem{theorem}{Theorem}[section]
\newtheorem{lemma}[theorem]{Lemma}
\newtheorem{proposition}[theorem]{Proposition}
\newtheorem{corollary}[theorem]{Corollary}
\newtheorem{definition}[theorem]{Definition}

\newtheoremstyle{customremark}
  {0.5em}
  {0.2em}
  {}
  {}
  {\scshape}
  {}
  {1ex}
  {}

\theoremstyle{customremark}
\renewenvironment{proof}{\par\noindent{\scshape Proof}\;}{\hfill\qedsymbol\par}
\newtheorem{remark}[theorem]{Remark}
\newtheorem{example}[theorem]{Example}

\biboptions{authoryear}

\begin{document}

\begin{frontmatter}

\title{Asymptotic uncertainty quantification for communities\\
  in sparse planted bi-section models\protect}

\author{B. J. K. Kleijn}
\address{Korteweg-de~Vries Institute for Mathematics,
    University of Amsterdam,
    P.O. Box 94248,
    1090 GE Amsterdam,
    The Netherlands}

\author{J. van Waaij\corref{correspondingauthor}}
\address{Department of Health Technology, Section of Bioinformatics, Technical University of Denmark,
Kemitorvet, 204, 252
2800 Kongens Lyngby
Denmark}


\cortext[correspondingauthor]{Corresponding author}
\ead{jvanwaaij@gmail.com}


\begin{abstract}
Posterior distributions for community structure
in sparse planted bi-section models are shown to achieve
exact (resp. almost-exact) recovery, with sharp bounds
for the sparsity regimes where edge probabilities decrease
as $O(\log(n)/n)$ (resp. $O(1/n)$). Assuming posterior
recovery, one may interpret credible sets (resp. enlarged
credible sets) as asymptotically consistent confidence
sets; the diameters of those credible sets are controlled
by the rate of posterior concentration. If credible levels
are chosen to grow to one quickly enough, corresponding
credible sets can be interpreted as frequentist confidence
sets without conditions on posterior concentration. In
the regimes with $O(1/n)$ edge sparsity, or when
within-community and between-community edge probabilities are
very close, credible sets may be enlarged to achieve
frequentist asymptotic coverage, also without conditions
on posterior concentration.
\end{abstract}


\end{frontmatter}


\section{Community detection and uncertainty quantification}
\label{sec:intro}

One of the central questions in network science concerns
community detection \citep{Girvan02}: one observes a graph with
vertices that belong to various (unobserved) communities and
edges that are present or not with community-dependent
probabilities. The goal is to infer the community structure
based on the presence or absence of edges in the observed graph.
The stochastic block model \citep{Holland83} is the most popular
model for the observation: edges in the observed random graph 
occur independently, with probabilities that depend on the
community membership of the vertices they connect. As such,
the stochastic block model is an inhomogeneous generalization
of the Erd\H os-R\'enyi model \citep{Erdos59}. These days
stochastic block models are applied in all branches of science
and its applications and are widely employed as canonical
models for the study of clustering and community structure
\citep{Fortunato10,Abbe18}.

Aside from its applications, the theory of the community
detection problem has attracted great attention from outside
network science, particularly from statistical physics, machine
learning, probability theory, combinatorics and statistics. The
machine learners' practically oriented perspective has led to a
wide range of algorithms for community detection, in which
computability is central. From the more stochastically centred
perspective of probabilists and statisticians, a large variety
of estimation methods for community structure has been proposed,
including spectral clustering (see \cite{Krzakala13} and many
others), maximization of the likelihood and other modularities
\cite{Girvan02,Bickel09,Choi12,Amini13}, semi-definite programming
\cite{Hajek16,Guedon16}, and penalized ML detection of communities
with minimax optimal misclassification ratio \cite{Zhang16,Gao17}).
Bayesian methods have been popular throughout, \eg, the original
work \cite{Snijders97}, the work of \cite{Decelle11a,Decelle11b},
with uniform priors and, more recently, \cite{Suwan16} with an
empirical prior choice. MCMC simulation of posterior distributions for community
structure is discussed, for example, in
\cite{mcdaid13,Geng19,Jiang21}). This very brief summary does not
do justice to the vast size and enormous variety of the literature
on community detection methods, and we refer to the highly
informative review of \cite{Abbe18} for an extensive bibliography
and a more comprehensive discussion. 

Community detection in very large graphs is used to assess
and compare detection methods: large numbers of edges supply large
amounts of information on community structure and community
detection methods should therefore be more accurate in large
graphs. Asymptotically, a natural requirement for any detection
method is consistency: as the number of vertices goes to infinity,
community estimates are required to coincide with (exact recovery)
or converge to (almost-exact recovery) the true, underlying
community structure with high probability. Let us denote the
observed graph by $X^n$, where $n$ is the number of vertices;
the (random) presence (resp. absence) of an edge between vertices
labelled $1\leq i,j\leq n$, $i\neq j$ is denoted $X^n_{ij}=1$
(resp. $X^n_{ij}=0$). In most variations of the stochastic block
model, the vertices belong exclusively to one of $K\geq2$
communities as described by the unobserved community assignment
vector $\theta_n$ with components $\theta_i\in\{0,\ldots,K-1\}$,
($1\leq i\leq n$). Community $k$, ($0\leq k\leq K-1$), has
$n_k=\sum_i 1_{\theta_i=k}$ members and vertex $i$,
($1\leq i\leq n$), has (random) degree $\sum_j X^n_{ij}$.
The edge connecting vertices $i$ and $j$ occurs with probability
$Q_n(\theta_i,\theta_j)$ depending on the communities of those
vertices. Stochastic block models vary in that they assume
known or unknown: the number of communities $K$, the edge
probabilities $Q_n:\{0,\ldots,K-1\}^2\to[0,1]$ and/or the
sizes $n_0,\ldots,n_{K-1}$ of the communities.

Note that the expected degree of vertex $i$ is equal to
$\sum_{j\neq i} Q_n(\theta_i,\theta_j)$, implying that
with $n$-independent edge-probabilities expected
degrees are proportional to the graph size $n$. Many
proofs of asymptotic consistency for community detection
methods are based on models in which $Q_n$ does not depend
on the graph size $n$, meaning that they describe
limits of stochastic block models \emph{with unbounded
degrees}. Real-world networks (\eg, social networks or
citation networks) describe vertices with expected degrees
that stay bounded or grow more slowly with the size of the
network. To describe large networks with bounded or
slowly-growing degrees, a form of \emph{edge sparsity}
is required: edge probabilities must decrease with
increasing graph size (this point is also emphasized in
\cite{Yuan22}). For example, if $Q_n=O(n^{-1})$, expected
degrees are bounded, and if $Q_n=O(n^{-1}\log(n))$,
expected degrees grow logarithmically with $n$. One may
then wonder \emph{which levels of edge sparsity make
community detection only just possible}; or conversely,
at which level of edge sparsity do community detection
and other forms of inference on the community structure
become impossible?

In \cite{Dyer89,Decelle11a,Decelle11b,Abbe16,Massoulie14,Mossel16}
and many other publications, feasibility of the community
detection problem and sharp bounds on edge sparsity are
studied in the context of the so-called \emph{planted
bi-section model}, which is a stochastic
block model with $K=2$ equally-sized communities of $n$
vertices each and edge probabilities $p_n$ (within
communities) and $q_n$ (between communities) that decrease
with $n$ (for a more detailed description of the model,
see \cref{sec:pbm}). The answers relate to the three
sparsity phases (\ie, fragmented, giant-component or connected)
of the Erd\H os-R\'enyi graph \citep{Erdos59,Bollobas07}: 
for example, \cite{Dyer89} showed that minimization of the
number of edges between estimated communities finds the true community
assignment vector with high probability, if there exists a
constant $A>0$ such that, $p_n-q_n\geq A n^{-1}\log n$; in
\cite{Mossel16} it is shown that community detection with
errors that converge to zero in probability is
possible, if and only if,
\begin{equation}
  \label{eq:MNSdetect}
  \frac{n(p_n-q_n)^2}{p_n+q_n}\to\infty,
\end{equation}
(see also \cite{Decelle11a,Decelle11b}). In
\cite{Massoulie14,Abbe16,Mossel15,Mossel16} it is shown that
if we write $p_n=a_nn^{-1}\log n$ and $q_n=b_nn^{-1}\log n$, 
assuming that $C^{-1}\leq a_n, b_n\leq C$, estimates coinciding
with the true community assignment vector with high probability
are possible, if and only if,
\begin{equation}
  \label{eq:mnscritical}
  (a_n+b_n-2\sqrt{a_nb_n}-1)\log n + \ft12\log\log n\to \infty.
\end{equation}
Conditions \cref{eq:MNSdetect} and \cref{eq:mnscritical} not
only lower-bound the degree of edge-sparsity, but also guarantee
sufficient distinction \citep{Janson10,Banerjee18} from the
Erd\H os-R\'enyi graph ($p_n=q_n$), in which communities are not
identifiable.

Besides community detection, other forms of inference on the
parameters defining a stochastic block model are studied. 
For example, \cite{Bickel16,Lei16} define asymptotically
consistent tests for the number of communities in a stochastic
block models with unbounded degrees. In \cite{Yuan22} an
asymptotically consistent likelihood ratio test is considered
to distinguish between the Erd\H os-R\'enyi graph in a
planted bi-section graph with bounded degrees.

The first goal of this paper is to explore the behaviour of
posterior distributions for the community assignment vector
in the planted bi-section model with bounded and slowly-growing
degrees and to demonstrate appropriate forms of \emph{posterior
consistency} under condition \cref{eq:MNSdetect} and
(a slight variation on) condition \cref{eq:mnscritical}
(see \cref{sec:pbmrecovery}).
The second, more important goal is \emph{frequentist uncertainty
quantification} based on an advantage that posteriors offer over
other estimation methods: in \cref{sec:pbmuncertainty},
Bayesian credible sets for community assignment are shown to
be (or can be enlarged to form) \emph{asymptotically consistent
confidence sets}. In \cref{sec:pbmconcdisc} we draw conclusions,
discuss some further possibilities and relate to other work.

Section \ref{sec:pbm} introduces the planted bi-section model, the
Bayesian posterior and test functions to prove its convergence.
Appendix \ref{app:defs} establishes notation and basic Bayesian
definitions; appendix \ref{app:PBMrc} introduces remote contiguity
and applies it to convert credible sets to confidence sets,
as in \cite{Kleijn21}.

\subsection*{Acknowledgements} The authors thank E.~Mossel and
J.~Neeman for a helpful discussion on necessary conditions for
exact recovery. BK thanks P.~Bickel for his encouragement to
pursue the confidence-sets-from-credible-sets question.

\section{The planted bi-section model, posteriors and tests}
\label{sec:pbm}

In this section, we introduce the model and prepare the
theorems on exact and almost-exact community detection in
the next section. We consider prior and posterior, define
metrics for community assignments, we derive posterior
concentration based on test functions and we prove the
existence of suitable test functions. 

\subsection{The planted bi-section model}
\label{sub:pbm} 

In a stochastic block model,
each vertex is assigned to one of $K\geq2$ communities through an
unobserved \emph{community assignment vector} $\theta'_n$.
Each vertex belongs to a community and any edge occurs (independently
of others) with a probability depending on the communities
of the vertices that it connects. In the \emph{planted bi-section
model}, there are only two communities ($K=2$) and, at the $n$-th
iteration ($n\geq1$), there are $2n$ vertices (labelled with
indices $1\leq i\leq 2n$), $n$ in each community,
with community assignment vector $\theta'_n\in\Theta_n'$ (with
components $\theta_{1}',\ldots,\theta_{2n}'\in\{0,1\}$), where
$\Theta'_n$ is the subset of $\{0,1\}^{2n}$ of all finite binary
sequences that contain as many ones as zeroes. Denote
that space in which the random graph $X^n$ takes its values by
$\scrX_n$ (\eg, represented by its adjacency matrix
with entries $\{X_{ij}:1\leq i,j\leq2n\}$) and its distribution by
$P_{\theta'_n,n}$. 
The ($n$-dependent) probability of an edge occuring ($X_{ij}=1$)
between vertices $1\leq i,j\leq 2n$ \emph{within the same community}
is denoted $p_n\in(0,1)$; the probability of an
edge \emph{between communities} is denoted $q_n\in(0,1)$,
\begin{equation}
  \label{eq:pbm}
  Q_n(\theta'_{n,i},\theta'_{n,j}):=P_{\theta'_n,n}(X_{ij}=1)=\begin{cases}
    \,\,p_n,&\quad\text{if $\theta_{n,i}'=\theta_{n,j}'$,}\\
    \,\,q_n,&\quad\text{if $\theta_{n,i}'\neq\theta_{n,j}'$.}
  \end{cases}
\end{equation}
Note that if $p_n=q_n$, $X^n$ is the Erd\H os-R\'enyi graph
$G(2n,p_n)$ and the community assignment $\theta'_n\in\Theta'_n$ is not
identifiable. Another identifiability issue that arises is that the model is 
invariant under interchange of community labels $0$ and $1$. This is
expressed in the parameter spaces $\Theta'_n$ through equivalence
relations: $\theta_{1,n}'\sim_n\theta_{2,n}'$, if
$\theta_{2,n}'=\neg\theta_{1,n}'$
(by component-wise negation). To prevent non-identifiability, we
parametrize the model for $X^n$ in terms of a parameter $\theta_n$
in a quotient space $\Theta_n=\Theta'_n/\sim_n$, for every $n\geq1$.
For \(\theta_n'\in\Theta_n'\) we denote the equivalence class
\(\{\theta_n',\neg\theta_n'\}\) by \(\theta_n\). Note that the
set \(\Theta_n\) can be identified with the set of partitions of
\(\set{1,\ldots,2n}\) consisting of exactly two sets with
$n$ elements, via the identification
\[
  \theta_n\longleftrightarrow\set{\set{i:\theta_{n,i}'=0},
  \set{i:\theta_{n,i}'=1}},
\]
Note that this is independent of the choice of the representation and
that $Q_n$ is well-defined on $\Theta_n\times\Theta_n$. 

Given true parameters $\theta_{0,n}\in\Theta_n$ ($n\geq1$), choose
representations \(\theta_{0,n}'\in\Theta_n'\) and 
define $Z_n(\theta_0')\subset\{1,\ldots,2n\}$ to be \emph{community zero}
(the set of all those $i$ such that $\theta_{0,i}'=0$) and call
the complement $Z^c_n(\theta_0')$ \emph{community one}. 
For the questions concerning exact recovery and detection, we
are interested in the sets $V_{n,k}'\subset\Theta_n'$, defined to
contain all those $\theta_n'$ that differ from $\theta_{0,n}'$ by
exactly $k$ exchanges of pairs: for $\theta_n'\in\Theta_n'$
we have $\theta_n'\in V_{n,k}'$, if
the set of vertices in community zero
\cf\ $\theta_{0,n}'$, $Z(\theta_{0,n}')=\{1\leq i\leq 2n: \theta_{0,n,i}'=0\}$,
from which we leave out the set of vertices in community zero
\cf\ $\theta_{n}'$, $Z(\theta_n')=\{1\leq i\leq 2n: \theta_{n,i}'=0\}$,
has $k$ elements.  Conversely, for any $\theta_{1,n}'$ and $\theta_{2,n}'$
in $\Theta_n'$, we denote the minimal number of pair-exchanges
necessary to take $\theta_{1,n}'$ into $\theta_{2,n}'$ by
$k'(\theta_{1,n}',\theta_{2,n}')$. Note that
\(k'(\theta_{1,n}',\neg\theta_{2,n}')=n-k'(\theta_{1,n}',\theta_{2,n}')\),
which leads to a metric on $\Theta_n$,
\begin{equation}
  \label{eq:definitionofk}
  k(\theta_{1,n},\theta_{2,n})=k'(\theta_{1,n}',\theta_{2,n}')
    \wedge k'(\theta_{1,n}',\neg\theta_{2,n}').
\end{equation}
Note that $k$ is independent of choice of the representations
and that $k$ takes values in \(\set{0,\ldots,\floor{n/2}}\). Now define,
\begin{equation}
\label{eq:definitionVnk}
  V_{n,k}=V_{n,k}(\theta_{0,n})=
    \set{\theta_n:k(\theta_n,\theta_{0,n})=k}=
      \set{\theta_n:\theta_n'\in V_{n,k}'},
\end{equation}
for $k\in\set{1,\ldots,\floor{n/2}}$.
Given some sequence $(k_n)$ of positive
integers we then define $V_n$ as the disjoint union,
\begin{equation}
  \label{eq:cupVnk}
  V_n = \bigcup_{k=k_n}^{\floor{n/2}} V_{n,k}.
\end{equation} 
Since we can choose two subsets of $k$ elements from two
sets of size $n$ in $\binom{n}{k}^2$ ways, the cardinality of $V_{n,k}$
is $\binom{n}{k}^2$, when \(k<n/2\) and \(\ft12\binom n{n/2}^2\)
when \(n\) is even and \(k=n/2\). In both cases the number of
elements in \(V_{n,k}\) is therefore bounded by \(\binom nk^2\). 

With that perspective on the parameter in mind, we note that the
likelihood with observed graph $X^n$ is given by,
\begin{equation}
  \label{eq:pbmlikelihood}
  p_{\theta,n}(X^n)=\prod_{i<j} Q_n(\theta_{n,i},\theta_{n,j})^{X_{ij}}
    (1-Q_n(\theta_{n,i},\theta_{n,j}))^{1-X_{ij}}.
\end{equation}
For the \emph{sparse versions} of the planted bi-section model,
we also define edge probabilities that vanish with growing $n$:
take $(a_n)$ and $(b_n)$ such that $a_n\log n=np_n$ and
$b_n\log n=nq_n$ for the so-called \emph{Chernoff-Hellinger phase};
take $(c_n)$ and $(d_n)$ such that $c_n=np_n$ and $d_n=nq_n$ for
the so-called \emph{Kesten-Stigum phase}. The fact that we do not
allow loops (edges that connect vertices with themselves) leaves
room for $2\cdot\ft12n(n-1)+n^2=2n^2-n=\ft12\cdot(2n)(2n-1)$
possible edges in the random graph $X^n$ observed at iteration
$n$.  

Our first statistical question of interest is reconstruction of
the (frequentist) true community assignment vectors $\theta_n$
\emph{consistently}, that is, (close to) correctly with
probability growing to one as $n$ tends to infinity. Consistency
can be formulated in various ways, and we consider two of those
formulations below.
\begin{definition}
\label{def:exact}
Let $\theta_{0,n}\in\Theta_n$ be given. An estimator sequence
$\hat{\theta}_n:\scrX_n\to\Theta_n$ is said to \emph{recover the
community assignment $\theta_{0,n}$ exactly} if,
\[
  P_{\theta_{0,n}}\bigl(\,\hat{\theta}_n(X^n)=\theta_{0,n}\,\bigr)\to1,
\]
as $n$ tends to infinity, that is, if $\hat{\theta}_n$ coincides
with the correct community assignment vector with high probability. 
\end{definition} 
We also relax this consistency requirement somewhat in the form
of the following definition, \cf\ \cite{Mossel16} and others: for
$n\geq1$ and two community assignments $\theta_{0,n},\theta_n\in\Theta_n$,
let $k(\theta_n,\theta_{0,n})$ denote the \emph{minimal number
of pair exchanges needed to
transform $\theta_n$ into $\theta_{0,n}$} (for further details, see
the definition of $k$, just before eq. (\ref{eq:cupVnk}) below).
\begin{definition}
\label{def:detect}
Let $\theta_{0,n}\in\Theta_n$ be given. An estimator sequence
$\hat{\theta}_n:\scrX_n\to\Theta_n$ is said to \emph{recover
$\tht_{0,n}$ almost-exactly}, if $k(\hat{\theta}_{n},\theta_{0,n})$
is of order $o_{P}(n)$ under $P_{\theta_{0,n}}$.
If, for some sequence $l_n=o(n)$,
\[
  P_{\theta_{0,n}}\bigl(\,
    k(\hat{\theta}_{n},\theta_{0,n}) \leq l_n
    \,\bigr)\to 1,
\]
as $n$ tends to infinity, we say that \emph{$\hat{\theta}_n$ recovers
$\theta_{0,n}$ with error rate $l_n$}.
\end{definition}

The second statistical problem we study is posterior-based, asymptotic,
frequentist uncertainty quantification for the community assignment
vector. We recall the central definition and its asymptotic version
for later reference.
\begin{definition}
\label{def:confset}
For fixed $n\geq1$ and some $0<a<1$, a set-valued map
$x^n\mapsto C(x^n)$ defined on $\scrX_n$ such that, for all
$\theta_n\in\Theta_n$, $P_{\theta_n,n}(\theta_n\in C(X^n))\geq 1-a$,
is called a confidence set of level $1-a$. If the levels $1-a_n$
of $n$-dependent confidence sets $C_n(X^n)$ go to $1$ as $n$
tends to infinity, the $C_n(X^n)$ are said to be asymptotically
consistent.
\end{definition}
Bayesian notions of uncertainty quantification, in particular
credible sets, and the way in which they are enlarged to form confidence
sets, is discussed in \cref{app:PBMrc}. 

\subsection{Prior, posterior and test functions}
\label{sub:pbmpriorposterior} 

Consider the sequence of experiments in which we observe random
graphs $X^n\in\scrX_n$ generated by the planted bi-section model
of definition \cref{eq:pbm}. In much of the literature on the
stochastic block model, the Bayesian approach is chosen: we pick
prior distributions $\pi_n$ for all $\Theta_n$, ($n\geq1$)
and calculate the posterior distribution: denoting the likelihood
by $p_{\theta,n}(X^n)$, the posterior for the parameter $\theta_n$ is
written as a fraction of sums, for all $A\subset\Theta_n$,
\[
  \Pi(A|X^n)={\displaystyle \sum_{\theta_n\in A}
    p_{\theta,n}(X^n)\, \pi_n(\theta_n)}
    \biggm/
  {\displaystyle \sum_{\theta_n\in\Theta_n}
    p_{\theta,n}(X^n)\, \pi_n(\theta_n)},
\]
where $\pi_n:\Theta_n\to[0,1]$ is the probability mass function for
the prior $\Pi_n$. Below we only consider \emph{uniform priors $(\Pi_n)$}
for $\theta_n\in\Theta_n$, so for all $n\geq1$ and $\theta_n\in\Theta_n$,
$\pi(\theta_n)= \pi_n:=(|\Theta_n|)^{-1}$.
\begin{remark}
To motivate our choice for a \emph{uniform prior}, note that for
community detection in the planted bi-section model, all values
of the community assignment are equivalent, in the sense that the
statistical problem is invariant under permutation of the vertices.
So, non-uniformity of the prior would imply a strictly subjective
bias, which has no place in the application of posteriors for
\emph{frequentist} inference. Additionally, non-uniformity of the
prior, although it would raise posterior concentration of mass at
particular values of the community assignment vector, would also
go at the expense of posterior mass at certain other values. Our
limits for posterior concentration are formulated \emph{for all
values of the community assignment}, so non-uniformity can only
make assertions weaker and conditions stronger.  
\end{remark}

According to lemma~2.2 in \cite{Kleijn21} (with $B_n=\{\theta_{0,n}\}$),
for any measurable sequence $\phi_{k,n}:\scrX_n\to[0,1]$
($k\geq1,n\geq1$) (following \cite{LeCam86}, we refer to such
functions as \emph{test functions} in what follows), we have,
\[
  \begin{split}
  P_{\theta_0,n}&\Pi(V_n|X^n)=\sum_{k=k_n}^{\floor{n/2}}
    P_{\theta_0,n}\Pi(V_{n,k}|X^n)\\
    &\leq \sum_{k=k_n}^{\floor{n/2}} 
    \Bigl( P_{\theta_0,n}\phi_{k,n}(X^n)
      + \sum_{\theta_n\in V_{n,k}}
      P_{\theta_n,n}(1-\phi_{k,n}(X^n)),
    \Bigr)
  \end{split}
\]
for every $n\geq1$. Suppose that for any $k\geq1$
there exists a sequence $(a_{n,k})_{n\geq1}$,
$a_{n,k}\downarrow0$ and, for any $\theta_n\in V_{n,k}$,
a test function $\phi_{\theta_n,n}$ that distinguishes $\theta_{0,n}$
from $\theta_n$ as follows,
\begin{equation}
  \label{eq:pttopttest}
  P_{\theta_0,n}\phi_{\theta_n,n}(X^n)+
    P_{\theta_n,n}(1-\phi_{\theta_n,n}(X^n)) \leq a_{n,k},
\end{equation}
for all $n\geq1$. Then using test functions 
$\phi_{k,n}(X^n)=\max\{\phi_{\theta_n,n}(X^n):\theta_n\in V_{n,k}\}$,
as well as the fact that,
\[
  P_{\theta_0,n}\phi_{k,n}(X^n)\le
    \sum_{\theta_n\in V_{n,k}} P_{\theta_0,n}\phi_{\theta_n,n}(X^n),
\]
we see that,
\begin{equation}
  \label{eq:pbmposteriorbound}
  \begin{split}
  P_{\theta_0,n}\Pi(V_n|X^n)
  & \leq \sum_{k=k_n}^{\floor{n/2}} \sum_{\theta_n\in V_{n,k}}
      \Bigl( P_{\theta_{0,n},n}\phi_{\theta_n,n}(X^n)
        +
      P_{\theta_n,n}(1-\phi_{\theta_n,n}(X^n))\Bigr)\\
  & \leq \sum_{k=k_n}^{\floor{n/2}} \binom{n}{k}^2 a_{k,n}.
  \end{split}
\end{equation}
This inequality forms the basis for the results in the next section
on exact recovery and almost-exact recovery.

\subsection{Existence of suitable test functions}
\label{sub:existenceoftests}

Given $n\geq1$ and two community assignment vectors
$\theta_{0,n},\theta_n\in\Theta_n$,
we are interested in calculation of the likelihood ratio
$dP_{\theta,n}/dP_{\theta_0,n}$, because it determines testing power
as well as the various forms of remote contiguity that
play a role.

Choose representations \(\theta_0'\) of \(\theta_0\) and
\(\theta'\) of \(\theta\) so that
\(k'(\theta_0',\theta')=k(\theta_0,\theta)\), where \(k\) and
\(k'\) are as in \cref{sec:pbmrecovery}. 
Recall that, $Z_n(\theta_0')\subset\{1,\ldots,2n\}$ is
community zero and the complement $Z^c_n(\theta_0')$ community one. For the sake
of presentation (in \cref{fig:thetanull} below), re-label the
vertices such that $Z(\theta_0')=\{1,\ldots,n\}$
and $Z^c(\theta_0')=\{n+1,\ldots,2n\}$. In the case $n=4$,
\cref{fig:thetanull} shows edge probabilities in the familiar
block arrangement.
\begin{figure}[bht]
 \begin{center}
   \parbox{0.75\textwidth}{\hspace*{-14ex}
     \begin{lpic}{thetas(0.75)}
       \lbl[t]{15,45;{$Z(\tht_{0,n}')$}}
       \lbl[t]{14,22;{$Z^c(\tht_{0,n}')$}}
       \lbl[t]{40,64;{$Z(\tht_{0,n}')$}}
       \lbl[t]{63,64;{$Z^c(\tht_{0,n}')$}}
       \lbl[t]{91,43;{$Z(\tht_n')$}}
       \lbl[t]{90,26;{$Z^c(\tht_n')$}}
       \lbl[t]{119,67;{$Z(\tht_n')$}}
       \lbl[t]{136,67;{$Z^c(\tht_n')$}}
     \end{lpic}
\caption{\label{fig:thetanull} Community assignments and edge
    probabilities according to $\tht_{0,n}'$ and to $\tht_n'$ for $n=4$
    and $k=1$. Vertex sets $Z(\cdot)$ and $Z^c(\cdot)$ correspond
    to communities zero and one for the given community assignment. Dark
    squares correspond to edges that occur with (within-community)
    probability $p_n$, and light squares to edges that occur with
    (between-community) probability $q_n$.}
  }
 \end{center}
\end{figure}

The likelihood under $\theta_0$ is given by equation
\cref{eq:pbmlikelihood}, with $\theta=\theta_0$. If we assume that
$\theta_{0,n}'$ and
$\theta_n'$ differ by $k$ pair-exchanges among respective members
of communities zero and one, then a look at \cref{fig:thetanull}
reveals that the likelihood-ratio depends only on the edges for
which exactly one of its end-points changes community. Define, 
\[
  \begin{split}
    A_n&=\{(i,j)\in\{1,\ldots,2n\}^2:\,i<j,\,\theta_{0,n,i}'=\theta_{0,n,j}',\,
      \theta_{n,i}'\neq\theta_{n,j}'\},\\
    B_n&=\{(i,j)\in\{1,\ldots,2n\}^2:\,i<j,\,\theta_{0,n,i}'\neq\theta_{0,n,j}'
      ,\,\theta_{n,i}'=\theta_{n,j}'\}.
  \end{split}
\]
Also define,
\[
  (S_n,T_n):=\Bigl(\sum\{X_{ij}:(i,j)\in A_n\},
    \sum \{X_{ij}:(i,j)\in B_n\}\Bigr),
\]
and note that the likelihood ratio can be written as,
\begin{equation}
  \label{eq:pbmlikratio}
  \frac{p_{\theta,n}}{p_{\theta_0,n}}(X^n)
  = \biggl(\frac{1-p_n}{p_n}\,\frac{q_n}{1-q_n}\biggr)^{S_n-T_n},
\end{equation}
where,
\begin{equation}
  \label{eq:SnTn}
  (S_n,T_n)\sim\begin{cases}
  \text{Bin}(2k(n-k),p_n)\times\text{Bin}(2k(n-k),q_n),
    \quad\text{if $X^n\sim P_{\theta_0,n}$},\\[2mm]
  \text{Bin}(2k(n-k),q_n)\times\text{Bin}(2k(n-k),p_n),
    \quad\text{if $X^n\sim P_{\theta,n}$}.
  \end{cases}
\end{equation}
Based on that, we derive the following lemma.
\begin{lemma}
\label{lem:testingpower}
Let $n\geq1$, $\theta_{0,n},\theta_n\in\Theta_n$ be given. Assume that
$\theta_{0,n}$ and $\theta_n$ differ by $k$ pair-exchanges. Then there
exists a test function $\phi_n:\scrX_n\to[0,1]$ such that,
\[
  P_{\theta_0,n}\phi_n(X^n) + P_{\theta,n}(1-\phi_n(X^n))\leq a_{n,k},
\]
with testing power,
\[
  a_{n,k}=\bigl( 1-p_n-q_n+2p_n\,q_n
    +2\sqrt{p_n(1-p_n)}\sqrt{q_n(1-q_n)})\bigr)^{2k(n-k)}.
\]
\end{lemma}
\begin{proof}
The likelihood ratio test $\phi_n(X^n)$ has testing power bounded by
the so-called Hellinger transform,
\[
  P_{\theta_0,n}\phi_n(X^n) + P_{\theta,n}(1-\phi_n(X^n))\leq
    \inf_{0\leq\al\leq1}
    P_{\theta_0,n}\Bigl(\frac{p_{\theta,n}}{p_{\theta_0,n}}(X^n)\Bigr)^\al,
\]
(see, \eg, \cite{LeCam86} and proposition~2.6 in \cite{Kleijn21}).
Using $\al=1/2$ (which is the minimum) and the independence of $S$ and
$T$, we find that,
\[
  P_{\theta_0,n}\biggl(\frac{p_{\theta,n}}{p_{\theta_0,n}}(X^n)\biggr)^{\ft12}
    = P_{\theta_0,n}
    \biggl(\frac{p_n}{1-p_n}\,\frac{1-q_n}{q_n}\biggr)^{\ft12(T_n-S_n)}
    = Pe^{\ft12\lambda_nS_n}\,Pe^{-\ft12\lambda_nT_n},
\]
where $\lambda_n:=\log(1-p_n)-\log(p_n)+\log(q_n)-\log(1-q_n)$ and
$(S_n,T_n)$ are distributed binomially, as in the first case of
\cref{eq:SnTn}. Using the moment-generating function of the
binomial distribution, we conclude that,
\[
  \begin{split}
  P_{\theta_0,n}&\biggl(\frac{p_{\theta,n}}
    {p_{\theta_0,n}}(X^n)\biggr)^{1/2}\\
  &= \biggl(\Bigl(1-p_n
    +p_n\Bigl(\frac{1-p_n}{p_n}\,\frac{q_n}{1-q_n}\Bigr)^{1/2}
    \Bigr)\\
  & \qquad\qquad\times\Bigl(1-q_n
      +q_n\Bigl(\frac{p_n}{1-p_n}\,\frac{1-q_n}{q_n}\Bigr)^{1/2}
      \Bigr)\biggr)^{2k(n-k)}\\
  &= \biggl(\Bigl(
    (1-p_n)+p_n^{1/2}q_n^{1/2}\Bigl(\frac{1-p_n}{1-q_n}\Bigr)^{1/2}
    \Bigr)\\
  & \qquad\qquad\times\Bigl((1-q_n)+p_n^{1/2}q_n^{1/2}
      \Bigl(\frac{1-q_n}{1-p_n}\Bigr)^{1/2}
      \Bigr)\biggr)^{2k(n-k)}\\[1mm]
  &= \Bigl((1-p_n)(1-q_n) + 2 \bigl(p_nq_n
    (1-p_n)(1-q_n)\bigr)^{1/2} + p_nq_n \Bigr)^{2k(n-k)},
  \end{split}
\]
which proves the assertion.
\end{proof}


\section{Exact and almost-exact posterior recovery of communities}
\label{sec:pbmrecovery}

In this section, we combine inequality \cref{eq:pbmposteriorbound}
with the test functions of subsection~\ref{sub:existenceoftests}
to arrive at two posterior concentration results, for exact and
almost-exact recovery of the community structure.

\subsection{Posterior consistency: exact recovery}
\label{sub:exacrrecovery}

For exact recovery, we are interested in the
expected posterior masses of subsets of $\Theta_n$ of the form:
\[
  V_n = \{ \theta_n\in\Theta_n: \theta_n\neq\theta_{0,n} \}
  = \bigcup_{k=1}^{\floor{n/2}} V_{n,k}.
\]
The theorem states a sufficient condition for $(p_n)$ and $(q_n)$,
which is related to requirement \cref{eq:mnscritical} in the
Chernoff-Hellinger phase.
\begin{theorem}
\label{thm:pbmexact}
For some $\theta_{0,n}\in\Theta_n$, assume that
$X^n\sim P_{\theta_0,n}$, for every $n\geq1$. If we equip
every $\Theta_n$ with its uniform prior and $p_n=a_nn^{-1}\log n$
and $q_n=b_nn^{-1}\log n$ are of order $O(n^{-1/2})$,
with $(a_n)$, $(b_n)$ such that,
\begin{equation}
  \label{eq:pbmpnqn}
 (a_n + b_n - 2\sqrt{a_nb_n} -1 ) \log n \to \infty,
\end{equation}
then,
\begin{equation}
  \label{eq:pbmexact}
  \Pi\bigl(\,\theta_n=\theta_{0,n}\bigm|X^n\,\bigr) \convprob{P_{\tht_{0,n}}} 1,
\end{equation}
as $n$ tends to infinity, \ie, the posterior recovers the community
assignment exactly.
\end{theorem}
\begin{proof}
We first give an alternative formulation of condition \cref{eq:pbmexact}
that is more suitable for the proof. Let $\delta_n = a_n+b_n-
2\sqrt{a_nb_n}-1=(\sqrt{a_n}-\sqrt{b_n})^2-1$, and \cref{eq:pbmexact}
implies that $\delta_n\log n\to\infty$. In particular, for large enough $n$,
$0<\delta_n<1$. Conversely, assume that, there is a sequence $\delta_n$
such that for $n$ sufficiently large, $0<\delta_n<1$,
$\delta_n\log n\to\infty$ as $n$ tends to infinity, and,
\begin{equation}
  \label{eq:equivalentcondition}
  (\sqrt{a_n}-\sqrt{b_n})^2 \geq 1+\delta_n, 
\end{equation}
for $n$ sufficiently large. Conclude that \cref{eq:pbmpnqn} is
equivalent to \cref{eq:equivalentcondition}.

According to \cref{lem:testingpower}, for every $n\geq1$, $k\geq1$
and given $\theta_{0,n}$, there exists a test sequence satisfying
\cref{eq:pttopttest} with $a_{n,k}=(1-\mu_n)^{2k(n-k)}$,
where,
\[
  \mu_n=p_n+q_n-2p_n\,q_n -2(p_n(1-p_n)q_n(1-q_n))^{1/2},
\]
in $[0,1]$. Start by noting that $\mu_n \geq p_n+q_n -2p_nq_n
-2\sqrt{p_n}\sqrt{q_n} = (\sqrt{p_n}-\sqrt{q_n})^2 -2p_nq_n$, 
so that assumption \cref{eq:equivalentcondition} implies that,
\begin{equation}
  1-\mu_n \leq 1-(1+\delta_n)\frac{\log n}{n}+\frac{2a_nb_n(\log n)^2}{n^2}.
\end{equation}
It follows from the assumption that $p_n$ and $q_n$ are of order
$O(n^{-1/2})$ and $\delta_n \log n\to\infty$, that that last term
on the right-hand side is smaller than $\frac12\delta_n\log n$
for large enough $n$. Therefore, $1-\mu_n \leq 1-(1+\delta_n/2)n^{-1}\log n$,
and \cref{lem:oneplusxdivrtothepowerrissmallerthanetothepowerx} says that,
\begin{equation}
    (1-\mu_n)^n \leq e^{-(1+\delta_n/2)\log n}. 
\end{equation}

Using that $\binom nk^2\leq \binom{2n}{2k}$ in
\cref{eq:pbmposteriorbound}, we find,
\begin{equation}
  \label{eq:twosums}
  \begin{split}
  P_{\theta_0,n}&\Pi(V_n|X^n)\\[2mm]
    &\leq\sum_{k=1}^{\floor{n/2}}\binom{n}{k}^2(1-\mu_n)^{2k(n-k)}
       \leq\sum_{k=1}^{\floor{n/2}}\binom{2n}{2k}(1-\mu_n)^{2k(n-k)}\\
    &\leq\sum_{\ell=2}^{2\floor{n/2}}\binom{2n}{\ell}(1-\mu_n)^{\ell(n-\ell/2)}
       \leq\sum_{k=1}^{n}\binom{2n}{k}(1-\mu_n)^{k(n-k/2)}\\
    &\leq\sum_{k=1}^{\floor{\delta_n n/2}}\binom{2n}{k}(1-\mu_n)^{k(n-k/2)}
       + \sum_{\ceil{\delta_n n/2}}^{n}\binom{2n}{k}(1-\mu_n)^{k(n-k/2)}.
  \end{split}
\end{equation}
We will analyse the two sums on the right-hand side separately.
Regarding the first sum, note that $n-k/2\geq (1-\delta_n/4)n$ for all
integers $1\leq k \leq \delta_n n/2$, so, 
\begin{equation}
  \begin{split}
    \sum_{k=1}^{\floor{\delta_n n/2}}&\binom{2n}{k}(1-\mu_n)^{k(n-k/2)}
      \leq\sum_{k=1}^{2n}\binom{2n}{k}(1-\mu_n)^{k(1-\delta_n/4)n}\\
    &\leq\sum_{k=1}^{2n}\binom{2n}k e^{-k(1-\delta_n/4)(1+\delta_n/2)\log n}  
    \leq\sum_{k=1}^{2n} \binom{2n}k e^{-k(1+\delta_n/8)\log n},
  \end{split}
\end{equation}
where we use that $(1-\delta_n/4)(1+\delta_n/2)\geq 1+\delta_n/8$.
By the binomial theorem and
\cref{lem:oneplusxdivrtothepowerrissmallerthanetothepowerx} the first
sum in \cref{eq:twosums} is bounded by,
\begin{equation}
  \label{eq:pbmfirstsum}
  \bigl(1+e^{-(1+\delta_n/8)\log n}\bigr)^{2n} - 1
    \leq e^{2ne^{-(1+\delta_n/8)\log n}} -1
    = e^{2e^{-(\delta_n/8)\log n}} -1\to 0,  
\end{equation}
as $n$ tends to infinity.

Regarding the second sum on the right-hand side of \cref{eq:twosums},
it is noted that $x(1-x)$ attains its minimum at $x=a$ on the interval
$[a,1/2], 0<a<1/2$, so that $k(n-k/2)=2n^2(k/(2n))(1-k/(2n))\geq
\delta_n(1-\delta_n/4)n^2/2$, for all integers $\delta_nn/2\leq k\leq n$.
Therefore, 
\begin{equation}
  (1-\mu_n)^{k(n-k/2)} \leq (1-\mu_n)^{\delta_n(1-\delta_n/4)n^2/2}
    \leq e^{-\frac12n\delta_n(1-\delta_n/4)(1+\delta_n/2)\log n }.
\end{equation} 
Substituting, we find, 
\begin{equation}
  \begin{split}
    \sum_{\ceil{\delta_n n/2}}^{2n} &\binom{2n}{k} (1-\mu_n)^{k(n-k/2)}
      \leq e^{-\frac12n\delta_n(1-\delta_n/4)(1+\delta_n/2)\log n }
        \sum_{k=0}^{2n}\binom{2n}k\\
    & = e^{2n\log 2-\frac12n\delta_n(1-\delta_n/4)(1+\delta_n/2)\log n}\to 0,
  \end{split} 
\end{equation}
as $n$ tends to infinity. The latter limit and \cref{eq:pbmfirstsum} prove
the assertion. 
\end{proof}

Up to the $\frac12\log\log(n)$-term, condition \cref{eq:pbmpnqn}
is equal to (the necessary and sufficient) condition \cref{eq:mnscritical}
of \cite{Mossel16}. In fact there is a trade-off: \cref{eq:pbmpnqn} is
slightly weaker than \cref{eq:mnscritical}, but \cref{eq:mnscritical}
applies only if there exists a $C>0$ such that $C^{-1}\leq a_n, b_n
\leq C$ for large enough $n$ \citep{Mossel16,Zhang16}. This bound
excludes some interesting examples in which one of the sequences
$(a_n)$ and $(b_n)$ may fade away with growing $n$ or equal zero
outright. For instance, if $b_n=0$ and $\liminf_na_n>1$, edges
between communities are completely absent but, separately, the
Erd\H os-R\'{e}nyi graphs spanned by vertices in $Z_n(\theta_0')$
and $Z^c_n(\theta_0')$ respectively are connected with high
probability. Similarly, if $a_n=0$ and $\liminf_nb_n>1$,
the posterior succeeds in exact recovery: possibly
with $b_n$ above $1$, edges between communities are abundant enough
to guarantee the existence of a path in $X^n$ that visits all
vertices at least once, with high probability. It is tempting to state
the following, well-known \citep{Abbe16,Abbe18,Mossel16} sufficient
condition for the sequences $a_n>0$ and $b_n>0$:
\begin{equation}
  \label{eq:pbmpnqnsimple}
  \text{\it $(\sqrt{a_n}-\sqrt{b_n})^2>c$,
  for some $c>1$ and $n$ large enough,}
\end{equation}
(even though it ignores the logarithm in \cref{eq:mnscritical}).
Figure~\ref{fig:phase} provides a `phase diagram' delineating the
choices for $(a,b)$ such that $a_n=a$ and $b_n=b$ leads to exact
recovery (analogous to \cite[theorem~3]{Abbe18}).
\begin{figure}[bht]
  \begin{center}
    \parbox{0.75\textwidth}{
      \hspace*{-5em}
      \begin{lpic}{Phase-diagram-exact-recovery(0.54)}
        \lbl[t]{30,78;{\Large $b$}}
        \lbl[t]{112,0;{\Large $a$}}
      \end{lpic}
      \caption{\label{fig:phase}The phase diagram for the sparse
      planted bi-section model in the Chernoff-Hellinger phase: values
      of $(a,b)$ that are too close to the diagonal, resemble the
      Erd\H os-R\'enyi model (in which there are no communities)
      too closely, rendering the community structure
      non-exactly-recoverable.}
    }
  \end{center}
\end{figure}
\begin{corollary}
Under the conditions of \cref{thm:pbmexact}, the
MAP-/ML-estimator recovers $\theta_{0,n}$ exactly.
\end{corollary}
\begin{proof}
Due to the uniformity of the prior, for every $n\geq1$,
maximization of the posterior density (with respect to the
counting measure) on $\Theta_n$, is the same as maximization of
the likelihood. Due to \cref{eq:pbmexact}, the posterior densities
in the points $\theta_{0,n}$ in $\Theta_n$ converge to one in
$P_{\theta_0,n}$-probability. Accordingly, the point of maximization
is $\theta_{0,n}$ with high probability.
\end{proof}

\subsection{Posterior consistency: almost-exact recovery}
\label{sub:almostexacrrecovery}

For the case of almost-exact recovery, the requirement
of convergence is less stringent: as said, \cite[proposition 2.9]%
{Mossel16} states that condition \cref{eq:MNSdetect} is
\emph{necessary and sufficient} for almost-exact recovery.
Below we show that posteriors with uniform priors \emph{recover}
the true community assignment \emph{almost exactly} if
\cref{eq:MNSdetect} holds.

We are interested in the expected posterior masses of subsets
of $\Theta_n$ of the form:
\[
  W_n = \bigcup_{k=k_n}^{\floor{n/2}} V_{n,k},
\]
for a sequence $k_n$ of order $o(n)$ or $O(n)$: the
posterior concentrates on community assignments $\theta_n$ that
differ from $\theta_{0,n}$ by no more than $k_n$ pair exchanges.
\begin{theorem}
\label{thm:pbmdetect}
For some $\theta_{0,n}\in\Theta_n$, let
$X^n\sim P_{\theta_0,n}$ for every $n\geq1$. If we equip
all $\Theta_n$ with uniform priors and edge-probabilities $(p_n)$,
$(q_n)$ and error rates $(k_n)$ are such that,
\begin{equation}
  \label{eq:pbmdetectcritical}
  \frac n{k_n}\Big(1-p_n-q_n+2p_n\,q_n
    +2\sqrt{p_n(1-p_n)q_n(1-q_n)}\Big)^{n/2}\to 0,
\end{equation}
as $n$ tends to infinity, then,
\begin{equation}
  \label{eq:pbmdetect}
  \Pi(W_n|X^n) \convprob{P_0} 0,
\end{equation}
as $n$ tends to infinity, \ie, the posterior recovers $\theta_{0,n}$
with error rate $k_n$.
\end{theorem}
\begin{proof}
According to \cref{lem:testingpower}, for every $n\geq1$, $k\geq1$
and given $\theta_{0,n}$, there exists a test sequence satisfying
\cref{eq:pttopttest} with $a_{n,k}=(1-\mu_n)^{2k(n-k)}$.
Therefore, using the inequalities
\(\binom{2n}k\leq \frac{(2n)^k}{k!}\) and \((n+m)!\geq n!m!\), the
Stirling lower bound formula, and finally our assumption
\(n(1-\mu_n)^{n/2}/k_n\to 0\) ($\mu_n$ as defined in the proof of
\cref{thm:pbmexact}, we see that for big enough \(n\),
\[
\begin{split}
  P_{\theta_0,n}\Pi(W_n&|X^n)
    \leq \sum_{k=k_n}^{\floor{n/2}}
      \binom{n}{k}^2 (1-\mu_n)^{2k(n-k)}\\[2mm]
    &\leq \sum_{k=2k_n}^n \binom{2n}{k} (1-\mu_n)^{k(n-k/2)}
    \leq \sum_{k=2k_n}^\infty \frac{1}{k!}  (2n)^k (1-\mu_n)^{kn/2}\\[2mm]    
    &\leq \frac{\bigl(2n(1-\mu_n)^{n/2}\bigr)^{2k_n}}{(2k_n)!}
      e^{2n(1-\mu_n)^{n/2}}.
\end{split}
\]
It then follows that,
\[
\begin{split}
  P_{\theta_0,n}\Pi(W_n&|X^n)
    \leq \frac1{\sqrt{4\pi k_n}}
      \biggl(\frac{n(1-\mu_n)^{n/2}}{k_n}\biggr)^{2k_n}
        e^{2k_n +  2n(1-\mu_n)^{n/2}}\\[2mm]
    &\leq\frac1{\sqrt{4\pi k_n}}\biggl(\frac{n(1-\mu_n)^{n/2}} {k_n}
      e^{1 + n(1-\mu_n)^{n/2}/k_n}\biggr)^{2k_n}\\[2mm]
    &\leq \frac{n(1-\mu_n)^{n/2}} {k_n}
      e^{1 + n(1-\mu_n)^{n/2}/k_n},
\end{split}
\]
which converges to zero as \(n\to\infty\). 
\end{proof}

\begin{example}
\label{ex:sparseexactrecovery}
Note that if \(p_n,q_n=O(n^{-1})=o(1),\) we may expand,
\[
  \sqrt{p_n}-\sqrt{q_n} =
  \frac1{2\sqrt{\frac12(p_n+q_n)}}(p_n-q_n)+O(|p_n-q_n|^2),
\]
which means that,
\[
  \mu_n=(\sqrt{p_n}-\sqrt{q_n})^2+O(n^{-2})
  =\frac{(p_n-q_n)^2}{2(p_n+q_n)}  + O(n^{-2}).
\]
Assuming only that $n(p_n-q_n)^2 > 2 (p_n+q_n)$, as in
\cite{Decelle11a,Decelle11b}, we would arrive at the
conclusion that $n\mu_n>1+O(n^{-1})$, which is insufficient in the proof
of \cref{thm:pbmdetect}. Note that a non-divergent choice
$k_n = O(1)$ forces us back into the Chernoff-Hellinger phase where
exact recovery is possible.
\end{example}
\begin{corollary}
\label{cor:pbmweakdetectcritical}
Under the conditions of \cref{thm:pbmdetect} with $(p_n)$ and
$(q_n)$ such that,
\begin{equation}
  \label{eq:pbmweakdetectcritical} n
  \Bigl(p_n+q_n-2p_n\,q_n
    -2\sqrt{p_n(1-p_n)q_n(1-q_n)} \Bigr)\to \infty,
\end{equation}
as $n$ tends to infinity, posteriors \emph{recover $\tht_{0,n}$ partially},
\[
  \Pi\bigl(\,k(\tht_n,\tht_{0,n})\geq \beta n\bigm| X^n\,\bigr)
    \convprob{P_0} 0,
\]
for \emph{any} fraction $\beta\in(0,\ft12)$, which implies that
the posterior recovers $\theta_{0,n}$ almost-exactly.
\end{corollary}
\begin{proof}
Let $\beta\in(0,\ft12)$ be given.
Follow the proof of \cref{thm:pbmdetect} with $k_n=\beta n$
and note that,
\[
  P_{\theta_0,n}\Pi\bigl(\,
    k(\tht_n,\tht_{0,n})\geq \beta n\bigm|X^n\,\bigr)
  \leq \frac1\beta(1-\mu_n)^{n/2} e^{1 + \beta^{-1}(1-\mu_n)^{n/2}}.
\]
Due to eq. \cref{eq:pbmweakdetectcritical},
\[
  (1-\mu_n)^{n/2}=\bigl(1-p_n-q_n+2p_n\,q_n
    +2\sqrt{p_n(1-p_n)q_n(1-q_n)}\bigr)^{n/2}\to 0,
\]
so $P_{\theta_0,n}\Pi(k(\tht_n,\tht_{0,n})\geq \beta n\bigm|X^n)\to0$.
For almost-exact recovery, let $\beta_m\downarrow0$ be given; if we
let $m(n)$ go to infinity slowly enough, posterior convergence
continues to hold with $\beta$ equal to $\beta_{m(n)}$.
\end{proof}
Condition \cref{eq:pbmweakdetectcritical} says that $n\mu_n\to\infty$
is \emph{sufficient} for almost-exact posterior recovery; but as
shown in \cite[proposition 2.10]{Mossel16}, it is \emph{also necessary}
for any form of almost-exact recovery. We conclude that if there
exist any estimators $\hat{\theta}_n$ that recover the
community assignment almost exactly, then posteriors with
uniform priors also recover the community assignment
almost-exactly.


\section{Uncertainty quantification for community structure}
\label{sec:pbmuncertainty}

Our first results on uncertainty quantification are obtained
with the help of the results in the previous section: if we know that
the sequences $(p_n)$ and $(q_n)$ satisfy requirements like
\cref{eq:pbmpnqn} or \cref{eq:pbmdetectcritical}, so that
exact or almost-exact recovery is guaranteed, then a consistent
sequence of confidence sets is easily constructed from credible
sets, as shown in \cref{sub:postconsconf} and the sizes of these
credible sets as well as the sizes of associated confidence sets
are controlled. If we cannot guarantee \cref{eq:pbmpnqn} or
\cref{eq:pbmdetectcritical}, or if we require explicit control over
confidence levels, confidence sets can still be constructed from
credible sets under conditions requiring that credible levels
grow to one quickly enough. Enlargement of credible sets may be
used to mitigate this condition, whenever we are close to the
Erd\H os-R\'enyi sub-model, as discussed in \cref{sub:credisconf}.

Regarding the sizes of credible sets, the most natural way to
compile a minimal-order credible set
$E_n(X^n)$ in a discrete space like $\Tht_n$, is to calculate the
posterior weights $\Pi(\{\tht_n\}|X^n)$ of all $\tht_n\in\Tht_n$,
order $\Tht_n$ by decreasing posterior weight into a finite sequence
$\{\tht_{n,1},\tht_{n,2},\ldots,\tht_{n,|\Tht_n|}\}$ and define
$E_n(X^n)=\{\tht_{n,1},\ldots,\tht_{n,m}\}$, for the smallest
$m\geq1$ such that $\Pi(E_n(X^n)|X^n)$
is greater than or equal to the required credible level.
To provide guarantees regarding the sizes of credible sets,
one would like to show that these $E_n(X^n)$
are of an order that is upper bounded with high probability.
(Although it is not so clear what the upper bound should be,
ideally.)

Here we shall follow a different path based on the smallest number
$k(\tht_n,\eta_n)$ of pair-exchanges between (two representations 
$\tht_n'$ and $ \eta'_n$ in $\Tht_n'$ of) \(\tht_n\) and \(\eta_n\)
respectively, see \cref{eq:definitionofk}. The map
$k:\Tht_n\times\Tht_n\to\{0,1,\ldots,
\left\lfloor{n/2}\right\rfloor\}$
is interpreted in a role similar to that of a metric on larger
parameter spaces: the \emph{diameter} $\diam_n(C)$ of a subset
$C\subset\Tht_n$ is,
\[
  \diam_n(C) = \max\bigl\{k(\tht_n,\eta_n):\tht_n,\eta_n\in C\bigr\}.
\]
by definition.

\subsection{Posterior recovery and confidence sets}
\label{sub:postconsconf}

If the posteriors concentrate amounts of
mass on $\{\tht_{0,n}\}$ arbitrarily close to one with growing $n$,
then a sequence of credible sets of a certain fixed level contains
$\tht_{0,n}$ for
large enough $n$. If such posterior concentration occurs with high
$P_{\tht_0,n}$-probability, then the sequence of credible sets is
also an asymptotically consistent sequence of confidence sets. We
formalize and prove this observation in the following theorem.
(A real-valued sequence $(c_n)$ is said to be {bounded away from
zero in the limit}, if $\liminf_nc_n>0$).
\begin{theorem}
\label{thm:postcredisconf}
Let $(c_n)$ be bounded away from zero in the limit. Suppose that
the posterior recovers the communities exactly. Then any sequence
$(D_n)$ of ($P_n^{\Pi}$-almost-sure) credible sets of levels
$c_n$ satisfies,
\[
  P_{\tht_0,n}\bigl( \,\tht_{0,n}\in D_n(X^n)\, \bigr)\to 1,
\]
\ie, $(D_n)$ is a consistent sequence of confidence sets.
Credible sets of minimal order (or diameter) equal $\{\tht_0\}$
with high $P_{\tht_{0,n}}$-probability. 
\end{theorem}
\begin{proof}
Note that with uniform priors $\Pi_n$,
$P_{\tht_0,n}\ll P_n^{\Pi}$ for all $n\geq1$, so that
$P_n^\Pi$-almost-surely defined credible sets $D_n$ of credible
level at least $\ep$, also satisfy,
\[ 
  P_{\tht_0,n}\bigl( \,\Pi(D_n(X^n)|X^n) \geq \ep \,\bigr) = 1.
\]
So if, in addition,
\[
  P_{\tht_0,n}\bigl( \,\Pi(\{\tht_{0,n}\}|X^n) > 1-\ep\,\bigr) \to 1,
\]
then $\tht_{0,n}\in D_n(X^n)$ with high $P_{\tht_0,n}$-probability.
Since all posterior mass is concentrated at $\tht_{0,n}$ with
high probability, the $\{\tht_{0,n}\}$ form a sequence of unique
credible sets of minimal order (or minimal diameter $k_n=0$) with
confidence levels greater than $\ep>0$ for large enough $n$.
\end{proof}
When only \emph{almost-}exact recovery is possible, the above
strategy to obtain confidence sets, carries over for
\emph{enlargements of credible sets}. Recall the definition
of the $V_{n,k}(\tht_n)$ in
\cref{eq:definitionVnk} (with \(\tht_{0,n}\) replaced by
\(\tht_n\)). Given some fixed underlying $\tht_{0,n}\in\Tht_n$,
we write $V_{n,k}$ for $V_{n,k}(\tht_{0,n})$. Making a certain
choice for the upper bounds $k_n\geq1$, we arrive at,
\begin{equation}
  \label{eq:differbyk}
  B_n(\tht_n)=\bigcup_{k=0}^{k_n} V_{n,k}(\tht_n),
\end{equation}
for every $n\geq1$ and $\tht_n\in\Tht_n$. Similar as for
\(V_{n,k}\) we write $B_n$ for $B_n(\tht_{0,n})$. Given a subset
$D_n$ of $\Tht_n$, the set $C_n\subset\Tht_n$ associated with $D_n$ under
$B_n(\tht_n)$ (see \cref{def:confcred}) then is the set
of $\tht_n\in\Tht_n$ whose \(k\)-distance from some element of $D_n$
is at most $k_n$,
\[
  C_n=\{\tht_n\in\Tht_n:\exists_{\eta_n\in D_n},\,k(\eta_n,\tht_n)
  \leq k_n\},
\]
the $k_n$-enlargement of $D_n$. If we
know that the sequences $(p_n)$ and $(q_n)$ satisfy \
requirement \cref{eq:pbmdetectcritical},
posterior concentration occurs around $\{\tht_{0,n}\}$ 
in `balls' of diameters $2k_n$ with growing $n$, and there exist
credible sets $D'_n$ of levels greater than $1/2$ and of diameters
$2k_n$ centred on $\tht_{0,n}$. The credible sets $D_n$ of
\emph{minimal diameters} of any level greater than $1/2$ must
intersect $D_n$. Then the $k_n$-enlargements $C_n$ of the
$D_n$ contain $\tht_{0,n}$.
\begin{theorem}
\label{thm:postconfsize}
Suppose that the posterior recovers communities almost-exactly
with error rate $(k_n)$,
\[
  \Pi\bigl(\, k(\tht_n,\tht_{0,n})\leq k_n\bigm|X^n\,\bigr)
    \convprob{P_{\tht_{0,n}}} 1.
\]
Let $(c_n)$ be bounded away from zero in the limit and let
$(D_n)$ denote a sequence of ($P_n^{\Pi}$-almost-sure)
credible sets of levels $c_n$. Then the $k_n$-enlargements
$C_n(X^n)$ of the $D_n(X^n)$ satisfy,
\[
  P_{\tht_0,n}\bigl( \,\tht_{0,n}\in C_n(X^n)\, \bigr)\to 1,
\]
\ie, the $k_n$-enlargements $(C_n)$ form a consistent sequence of
confidence sets. If the sets $D_n$ have \emph{minimal diameters},
then,
\[
  \diam_n(D_n(X^n))\leq2k_n,\quad\diam_n(C_n(X^n))\leq4k_n,
\]
with high $P_{\tht_0,n}$-probability.
\end{theorem}
\begin{proof}
As in the proof of \cref{thm:postcredisconf},
$P_n^\Pi$-almost-surely defined credible sets $D_n$ of credible
level at least $c_n$ also satisfy,
\[ 
  P_{\tht_0,n}\bigl( \,\Pi(D_n(X^n)|X^n) \geq c_n \,\bigr) = 1.
\]
Convergence of the posterior implies
that with growing $n$, the balls $B_n(\tht_{0,n})$ of radii $k_n$
centred on $\tht_{0,n}$ contain an arbitrarily large fraction of
the total posterior mass, so assuming that $n$ is large enough,
$c_n>\ep>0$ and $\Pi(B_n(\tht_{0,n})|X^n)>1-\ep$ with high
$P_{\tht_{0,n}}$-probability. Conclude that,
\[
  B_n(\tht_{0,n})\cap D_n(X^n) \neq\emptyset,
\] 
with high $P_{\tht_{0,n}}$-probability, which amounts
to asymptotic coverage of $\tht_{0,n}$ for the
$k_n$-enlargement $C_n(X^n)$ of $D_n(X^n)$.
Now fix $n\geq1$. For every $\tht_n\in\Tht_n$ and every
$x^n\in\scrX_n$, let $k_n(\tht_n,x^n)$ denote the minimal
radius of balls $B$ in $\Tht_n$ centred on $\tht_n$ of posterior mass
$\Pi(B|x^n)\geq c_n$. Let $\hat{\tht}_n(x^n)\in\Tht_n$ be
such that,
\[
  k_n(\hat{\tht}_n(x^n))
    =\min\bigl\{ k_n(\tht_n,x^n):\tht_n\in\Tht_n \bigr\},
\]
\ie, the centre point of a \emph{smallest} level-$c_n$
credible ball in $\Tht_n$. To conclude, note that
$k_n(\hat{\tht}_n(X^n))\leq k_n$ with high $P_{\tht_{0,n}}$-probability
and if the $D_n(X^n)$ are of minimal diameters, then they
are contained in $k_n(\hat{\tht}_n(X^n))$-balls centred
on some $\hat{\tht}_n(X^n)$. 
\end{proof}

\subsection{Confidence sets directly from credible sets}
\label{sub:credisconf}

To use \cref{thm:postcredisconf}
or \cref{thm:postconfsize}, the statistician needs
to know that the sequences $(p_n)$ and $(q_n)$ satisfy \cref{eq:pbmpnqn}
or \cref{eq:pbmdetectcritical}, basically to satisfy the testing
condition \cref{eq:pttopttest}. Particularly,
condition \cref{eq:pbmweakdetectcritical} is \emph{not} strong
enough to apply \cref{thm:postconfsize}.
But even if that knowledge is not
available and testing cannot serve as a condition, the use of
credible sets as confidence sets remains valid, as long as
credible levels grow to one fast enough. The following proposition
also provides lower bounds for confidence levels of credible sets.
\begin{proposition}
\label{prop:credisconf}
Let $\theta_{0,n}$ in $\Theta_n$ with uniform priors $\Pi_n$,
$n\geq1$, be given and define $b_n=|\Tht_n|^{-1}=
(\ft12\binom{2n}{n})^{-1}$. Let $D_n$ be a sequence of credible sets,
such that,
\[
  \Pi(D_n(X^n)|X^n)\geq 1-a_n,
\]
for some sequence $(a_n)$ with $a_n=o(b_n)$. Then,
\[
  P_{\theta_0,n}\bigl(\,\theta_0\in D_n(X^n)\,\bigr)
  \geq 1-b_n^{-1}a_n.
\]
\end{proposition}
\begin{proof}
If $\tht_{0,n}\not\in D_n(X^n)$ then $\Pi(\{\tht_{0,n}\}|X^n)\leq a_n$,
$P_n^{\Pi}$-almost-surely. Then,
\[
  \begin{split}
  P_{\tht_0,n}\bigl(&\tht_0\in\Tht\setminus D_n(X^n)\bigr)
    = P_n^{\Pi|\{\tht_0\}}\bigl(\tht_0\in\Tht\setminus D_n(X^n)\bigr)\\
  &= b_n^{-1}\int_{\{\tht_{0,n}\}}
      P_{\tht,n}\bigl(\tht_0\in \Tht\setminus D_n(X^n)\bigr)
      \,d\Pi_n(\tht)\\
  &= b_n^{-1}P_n^\Pi \bigl( 1\{\tht_0\in\Tht_n\setminus D_n(X^n)\}
      \,\Pi(\{\tht_{0,n}\}|X^n)\bigr)
  \leq b_n^{-1}a_n,
  \end{split}
\]
by Bayes's Rule \cref{eq:disintegration}.
\end{proof}

\Cref{thm:coverage} leaves room for mitigation of the lower
bound on credible levels if we are willing to use enlarged
credible sets. There are two competing
influences when enlarging: on the one hand, the prior
masses $b_n=\Pi_n(B_n(\theta_{0,n}))$ become larger, relaxing the
lower bounds for credible levels. On the
other hand, enlargement leads to likelihood ratios with random
fluctuations that take them further away from one (see
\cref{lem:rcfirstlemmaprime,lem:pbmdn}), thus
interfering with notions like contiguity and remote contiguity
(see \cref{app:PBMrc} and \cite{Kleijn21}).
Whether \cref{prop:credisconf} is useful and whether
enlargement of credible sets helps, depends on the sequences
$(p_n)$ and $(q_n)$: we consider the situation in which edge
differences between within-community and between-community edge
probabilities become less-and-less pronounced:
\begin{equation}
  \label{eq:pbmphaseone}
  p_n-q_n=o\bigl(n^{-1}\bigr),
\end{equation}
while satisfying also the condition that, 
\begin{equation}
  \label{eq:pbmphasetwo}
  p_n^{1/2}(1-p_n)^{1/2}
      +q_n^{1/2}(1-q_n)^{1/2}=o\bigl(n|p_n-q_n|\bigr).
\end{equation}
In this regime either $p_n,q_n\to0$ or $p_n,q_n\to1$, signifying
sparsity of either presence or absence of edges respectively. 
If $p_n,q_n\to0$, \cref{eq:pbmphasetwo} amounts to,
\begin{equation}
  \label{eq:pbmphasetwoprime}
  n(p_n^{1/2}-q_n^{1/2})\to \infty,
\end{equation}
so differences between $p_n$ and $q_n$ may not converge to zero
too fast (essentially in order to maintain sufficient distinction
from the Erd\H os-R\'enyi graph \cite{Janson10,Banerjee18}).
For the following lemma we define,
\[
  \rho_n
    = \min\biggl\{\Bigl(\frac{1-p_n}{p_n}\,\frac{q_n}{1-q_n}\Bigr),
      \Bigl(\frac{p_n}{1-p_n}\,\frac{1-q_n}{q_n}\Bigr)\biggr\}
    =e^{-|\lambda_n|},
\]
where $\lambda_n:=\log(1-p_n)-\log(p_n)+\log(q_n)-\log(1-q_n)$, and,
\[
  \alpha_n = \int 2k(\theta_{0,n},\theta_n)(n-k(\theta_{0,n},\theta_n))
      \,d\Pi_n(\theta_n|B_n)
    = \frac1{|B_n|}\sum_{k=0}^{k_n}\binom{n}{k}^2 2k(n-k),
\]
with the following rate for remote contiguity (see
\cref{def:remctg}):
\begin{equation}
  \label{eq:pbmdn}
  d_n = \rho_n^{C\alpha_n|p_n-q_n|},
\end{equation}
for some $C>1$.

For the following lemma, let $(p_n)$ and $(q_n)$ be given (with
resulting $(\rho_n)$), and let $C>$, $(k_n)$ and
$\theta_{0,n}\in\Theta_n$, for all $n\geq1$ be given (with
resulting $(\alpha_n)$ and $(d_n)$).
\begin{lemma}
\label{lem:pbmdn}
Assume that \cref{eq:pbmphasetwo} holds.
Then,
\[
  P_{\theta_0,n} \ctg d_n^{-1} P_n^{\Pi|B_n},
\]
with $B_n=B_n(\theta_{0,n})$ like in \cref{eq:differbyk}.
\end{lemma}
\begin{proof}
Let $(k_n)$ and $\theta_{0,n}\in\Theta_n$ be given. We denote
$P_n=P_n^{\Pi|B}$, $Q_n=P_{\theta_0,n}$ and apply Jensen's
inequality to obtain,
\[
  \begin{split}
  \frac{dP_n}{dQ_n}(X^n)
    &= \frac1{|B_n|}
      \sum_{\theta_n\in B_n}
      \biggl(\frac{1-p_n}{p_n}\,\frac{q_n}{1-q_n}
        \biggr)^{S_n(\theta_n)-T_n(\theta_n)}\\
    & \geq \exp\Bigl(\frac{\lambda_n}{|B_n|}
      \sum_{\theta_n\in B_n} \bigl(S_n(\theta_n)-T_n(\theta_n)\bigr)\Bigr),
  \end{split}
\]
where $(S_n(\theta_{n}),T_n(\theta_{n}))$ is distributed as in
\cref{eq:SnTn}. By invariance of the sum under permutations
of the vertices, we re-sum as follows for any $k\geq1$,
\[
  \frac1{|V_{n,k}|}\sum_{\theta_n\in V_{n,k}} S_n(\theta_n)
    = \frac{2k(n-k)}{n(n-1)}S_n,\,\,
  \frac1{|V_{n,k}|}\sum_{\theta_n\in V_{n,k}} T_n(\theta_n)
    = \frac{2k(n-k)}{n^2}T_n,
\]
where, with the notation $Z_n=Z(\theta_{0,n}')\subset\{1,\ldots,2n\}$,
for a certain representation \(\theta_{0,n}'\) of \(\theta_{0,n}\),
for the zero elements of \(\theta_{0,n}'\),
\[
  \begin{split}
  S_n &= \sum_{i,j\in Z_n}X_{ij}
    + \sum_{i,j\in Z_n^c}X_{ij}\sim\text{Bin}(n(n-1),p_n),\\
  T_n &= \sum_{i\in Z_n,\,j\in Z^c}X_{ij}
    + \sum_{i\in Z_n^c,\,j\in Z}X_{ij}\sim\text{Bin}(n^2,q_n),
  \end{split}
\]
which gives us the lower bound,
\[
  \frac{dP_n}{dQ_n}(X^n)
  \geq \rho_n^{\sum_{k=0}^{k_n}2k(n-k)\frac{|V_{n,k}|}{|B_n|}
    |\bar{S}_n-\bar{T}_n|}
  = \rho_n^{\alpha_n|\bar{S}_n-\bar{T}_n|},
\]
where $\bar{S}_n=S_n/(n(n-1))$ and $\bar{T}_n=T_n/n^2$. By the central
limit theorem,
\[
  \Biggl( \frac{n(\bar{S}_n-p_n)}{p_n^{1/2}(1-p_n)^{1/2}},
    \frac{n(\bar{T}_n-q_n)}{q_n^{1/2}(1-q_n)^{1/2}} \Biggr)
    \convweak{Q_n} N(0,1)\times N(0,1),
\]
which implies that for every $\ep>0$ there exists an $M>0$ such
that,
\[
  \sup_{n\geq1} Q_n\Biggl(
    \frac{n(\bar{S}_n-p_n)}{p_n^{1/2}(1-p_n)^{1/2}}\vee
    \frac{n(\bar{T}_n-q_n)}{q_n^{1/2}(1-q_n)^{1/2}} > M \Biggr)<\ep.
\]
Conclude that,
\[
  \sup_{n\geq1} Q_n\biggl( \Bigl(\frac{dP_n}{dQ_n}\Bigr)^{-1} \leq
    \rho_n^{-\alpha_n\bigl(\frac Mn(p_n^{1/2}(1-p_n)^{1/2}
      +q_n^{1/2}(1-q_n)^{1/2}) + |p_n-q_n|\bigr)} \biggr) \geq 1 - \ep.
\]
Note that the term in the exponent proportional to $M$ is
dominated by $|p_n-q_n|$ by \cref{eq:pbmphasetwo}.
Hence for every $C>1$ and every $\ep>0$,
\[
  Q_n\biggl( \Bigl(\frac{dP_n}{dQ_n}(X^n)\Bigr)^{-1}
    \leq \rho_n^{-C\alpha_n |p_n-q_n|} \biggr) \geq 1 - \ep,
\]
for large enough $n$.
Using the remark following \cref{lem:rcfirstlemmaprime},
we see that $P_{\theta_0,n} \ctg d_n^{-1} P_n^{\Pi|B_n}$,
with $d_n$ as in \cref{eq:pbmdn}.
\end{proof}
This argument amounts to a proof for the following theorem (immediate
from \cref{thm:coverage}).
\begin{theorem}
\label{thm:pbmsparseconfsets}
Let $(k_n)$ be given and assume that $(p_n)$ and $(q_n)$ satisfy
\cref{eq:pbmphaseone} and \cref{eq:pbmphasetwo} (or
$p_n,q_n\to0$ and \cref{eq:pbmphasetwoprime}). Let
$\theta_{0,n}$ in $\Theta_n$ with uniform priors $\Pi_n$
be given and let $D_n(X^n)$ be a sequence of credible sets of credible
levels $1-a_n$, for some sequence $(a_n)$ such that
$b_n^{-1}a_n = o(d_n)$. Then the sets $C_n(X^n)$, associated with
$D_n(X^n)$ under $B_n$ as in \cref{eq:differbyk} satisfy,
\[
  P_{\theta_0,n}\bigl(\,\theta_0\in C_n(X^n)\,\bigr)
    \to1,
\]
\ie, the $C_n(X^n)$ are asymptotic confidence sets.
\end{theorem}
Consider the possible choices for $(a_n)$ if we assume
$k_n=\beta\,n$ for some fixed $\beta\in(0,\ft12)$
(as in the proof of \cref{cor:pbmweakdetectcritical}).
First of all, Stirling's approximation gives rise to the
following approximate lower bound on the factor between
prior mass and prior mass without enlargement:
\[
  \frac{\Pi_n(B_n)}{\Pi_n(\{\tht_{0,n}\}))}
    =\sum_{k=0}^{k_n}\binom{n}{k}^2
    \geq \binom{n}{k_n}^2
    \geq \frac1{2\pi n}\frac1{\beta(1-\beta)}f(\beta)^n,
\]
where $f:(0,\ft12)\to(1,4)$ is given by,
\[
  f(\beta) = (1-\beta)^{-2(1-\beta)}\beta^{-2\beta}.
\]
Approximating $\alpha_n\approx 2k_n(n-k_n)$ for large $n$ and
using \cref{eq:pbmphaseone}, we also have,
\[
  d_n = \rho_n^{C\alpha_n|p_n-q_n|} \approx
    \rho_n^{2Cn^2\beta(1-\beta)|p_n-q_n|}
    = e^{-|\lambda_n|o(n)}.
\]
So if we assume that $\lambda_n=O(1)$, $d_n$ is sub-exponential and
does not play a role for the improvement factor.

Conclude as follows: (let $a_n=o(|\Tht_n|^{-1})\approx o(4^{-n})$
denote the rates appropriate in \cref{prop:credisconf}
and assume $\lambda_n=O(1)$) if we have
credible sets $D_n(X^n)$ of credible levels $1-a_n f(\beta)^{n(1+o(1))}$,
then the sequence of enlarged confidence sets $(C_n(X^n))$,
associated with $D_n(X^n)$ through $B_n$ with $k_n=\beta n$,
covers the true value of the community assignment parameter with
high probability. Credible levels that had to be
of order $1-a_n\approx 1-o(4^{-n})$ previously, can be
of approximate order $1-o(c^{-n})$ for any $1<c<4$ by
enlargement by $B_n$ if conditions
\cref{eq:pbmphaseone} and \cref{eq:pbmphasetwo} hold; the closer
$0<\beta<\ft12$ is to $\ft12$, the closer $c$ is to $1$.


\section{Conclusions and discussion}
\label{sec:pbmconcdisc}

In this paper we consider application of Bayesian posteriors for
frequentist asymptotic inference on the community structure of
sparse planted bi-section graphs. More specifically, we prove that
the posterior recovers the true community assignment (exactly,
respectively almost-exactly) in the sparsest possible
(Chernoff-Hellinger, respectively Kesten-Stigum) cases.

We also use the posterior concentration results to draw conclusions
regarding the role of (enlarged) Bayesian credible sets as
frequentist asymptotic confidence sets. For this purpose, it is
important that the posterior concentration results are sharp;
otherwise the credible sets we choose are too
large and the required confidence levels are too high, leading
to asymptotic confidence sets that are too conservative.

The analysis we give is limited in several respects. First of
all, although realistic regarding expected degrees in large
graphs, the planted bi-section model is a highly stylized random
graph model; more flexible is the family of stochastic block
models, which leaves room for more than two communities, and
classes of unequal sizes (and generalizations thereof). On
uncertainty quantification in stochastic block models, the
literature is very limited: in \cite{vanWaaij20}, the present
analysis is extended to an unknown number of communities of
order $O(\sqrt{n/\log(n)})$, of unknown sizes bounded above
and below proportional to the graph size $n$.

Practical implementation of what we propose is not straightforward:
with a graph $X^n$ of fixed size $n$, a bound like
\cref{eq:pbmposteriorbound} in combination with
\cref{lem:testingpower} (both of which hold for finite $n$)
permits calculation of a lower bound for $k_n$. With the
corresponding enlargement radius, confidence sets can then be
constructed as explicit radius-$k_n$-enlargements of credible
sets from the posterior. This finite-$n$ programme
is followed in \cite{Kleijn21b}.

To conclude we note that Bayesian methods may have lost some
of their popularity of late, because the computational burden
of sampling a posterior distribution is deemed relatively high.
In the planted bi-section model, for example, other more
efficient methods for the recovery of the community
structure in the planted bi-section model exist. Based on the
preceding, however, we argue that if \emph{uncertainty quantification}
for the community structure is the goal, the relatively high
computational cost of simulating a posterior is justifiable.
Since (limiting) sampling distributions of other estimators are
prohibitively hard to obtain or analyse, constructing asymptotic
confidence sets for community structure in other ways may prove
to be very hard or even impossible.


\appendix

\renewcommand{\thesection}{\Alph{section}}


\section{Definitions and conventions\label{app:defs}}

We assume given for every $n\geq1$, a random graph
$X^n$ taking values in the (finite) space
$\scrX_n$ of all undirected graphs with $n$ vertices and no self-loops.
We denote the powerset of $\scrX_n$ by $\scrB_n$ and regard it as
the domain for probability distributions
$P_n:\scrB_n\to[0,1]$ a model $\scrP_n$ parametrized by
$\Theta_n\rightarrow\scrP_n:\theta\mapsto P_{\theta,n}$
with finite parameter spaces $\Theta_n$ (with powerset $\scrG_n$)
and uniform priors
$\Pi_n$ on $\Tht_n$. As frequentists, we assume that there exists
a `true, underlying distribution for the data' $P_{0,n}$; in this
case, that means that for every $n\geq1$, there exists a
$\tht_{0,n}\in\Tht_n$ and corresponding $P_{\tht_0,n}$ from which
the $n$-th graph $X^n$ is drawn.
\begin{definition}
Given $n\geq1$ and a prior probability measure $\Pi_n$ on
$\Tht_n$, define the \emph{$n$-th prior predictive distribution}
as:
\begin{equation}
  \label{eq:priorpred}
  P_n^{\Pi}(A) = \int_\Theta P_{\theta,n}(A)\,d\Pi_n(\theta),
\end{equation}
for all $A\in\scrB_n$. For any $B_n\in\scrG_n$ with $\Pi_n(B_n)>0$,
define also the \emph{$n$-th local prior predictive distribution},
\begin{equation}
  \label{eq:localpriorpred}
  P_n^{\Pi|B}(A) = \frac{1}{\Pi_n(B_n)}\int_{B_n}
    P_{\theta,n}(A)\,d\Pi_n(\theta),
\end{equation}
as the predictive distribution on $\scrX_n$ that results from the prior
$\Pi_n$ when conditioned on $B_n$.
\end{definition}
The prior predictive distribution $P_n^{\Pi}$ is the marginal
distribution for $X^n$ in the Bayesian perspective that
considers parameter and sample jointly
$(\theta,X^n)\in\Theta\times\scrX_n$
as the random quantity of interest. 
\begin{definition}
\label{def:posterior}
Given $n\geq1$,
a \emph{(version of) the posterior} is any set-function
$\scrG_n\times\scrX_n\rightarrow[0,1]:
(A,x^n)\mapsto\Pi(\,A\,|X^n=x^n)$
such that,
\begin{enumerate}
\item for $B\in\scrG_n$, the map
  $x^n\mapsto\Pi(B|X^n=x^n)$ is
  $\scrB_n$-measurable,
\item for all $A\in\scrB_n$ and $V\in\scrG_n$,
  \begin{equation}
    \label{eq:disintegration}
    \int_A\Pi(V|X^n)\,dP_n^{\Pi} = 
    \int_V P_{\theta,n}(A)\,d\Pi_n(\theta).
  \end{equation}
\end{enumerate}
\end{definition}
Bayes's Rule is expressed through equality \cref{eq:disintegration}
and is sometimes referred to as a `disintegration' (of the joint
distribution of $(\theta,X^n)$).

Because we take the perspective of a frequentist using Bayesian
methods, we are obliged to demonstrate that Bayesian definitions
continue to make sense under the assumption that the data $X^n$ is
distributed according to a true, underlying $P_{0,n}$: Bayesian concepts
above that have been defined through conditioning, are almost-sure
with respect to the relevant marginal. In the case of $X^n$, the
relevant marginal is the prior predictive distribution. Accordingly,
we have to assume that $P_{\theta_0,n}\ll P^{\Pi}_n$, for all $n\geq1$.

The following lemma is used in two places in the text. 
\begin{lemma}
\label{lem:oneplusxdivrtothepowerrissmallerthanetothepowerx} 
For all positive integers \(r\) and real numbers \(x>-r,\) \((1+x/r)^r
\leq e^x\).
\end{lemma}
\begin{proof}
Let for \(x>-r\), \(f(x)= r\log(1+x/r)\) and \(g(x)=x\). Then
\(f'(x)=(1+x/r)^{-1}\) and \(g'(x)=1\). Then \(f'(x)\leq g'(x)\),
when \(x\geq 0\), \(f'(x)>g'(x) \) when \(-n<x<0\) and \(f(0)=g(0)\).
It follows that \(f(x)\leq g(x)\) for all \(x>-r\).
As \(y\to e^y\) is increasing for all real \(y\), we find \(x>-n\),
\((1+x/r)^r = e^{f(x)}\leq e^{g(x)}=e^x\).
\end{proof}

\subsubsection*{Notation and conventions}

Asymptotic statements that end in ``... with high probability''
indicate that said statements are true with probabilities that
grow to one.
For given probability measures $P,Q$ on a measurable space
$(\Omega,\scrF)$, we define
the Radon-Nikodym derivative $dP/dQ:\Omega\to[0,\infty)$,
$P$-almost-surely, referring {\it only} to the $Q$-dominated
component of $P$, following \cite{LeCam86}. We also \emph{define}
$(dP/dQ)^{-1}:\Omega\to(0,\infty]:\omega\mapsto 1/(dP/dQ(\omega))$,
$Q$-almost-surely.
Given random variables $Z_n\sim P_n$, weak convergence to a random
variable $Z$ is denoted by $Z_n\convweak{P_n}Z$, convergence
in probability by $Z_n\convprob{P_n}Z$ and almost-sure convergence
(with coupling $P^\infty$) by $Z_n\convas{P^{\infty}}Z$.
The integral of a
real-valued, integrable random variable $X$ with respect to a
probability measure $P$ is denoted $PX$, while integrals over
the model with respect to priors and posteriors are always written
out in Leibniz's or sum notation. The cardinality of a set $B$ is
denoted $|B|$.


\section{Remote contiguity and confidence sets}
\label{app:PBMrc}

Bayesian asymptotics has seen a great deal of development over
recent decades, but the essence of the theory remains
that of Schwartz's theorem: a balance between testing power and
a minimum of prior mass `locally', leads to a controlled limit
for the posterior distribution with a frequentist interpretation.
It has also become clear that the same notion of `locality'
allows conversion of sequences of credible sets to asymptotic
confidence sets and that is the purpose of this
paper as well. `Locality' in the above sense is defined through
a weakened form of contiguity called remote contiguity
\citep{Kleijn21}. In this appendix, we summarize these points,
to support the proofs of \cref{lem:pbmdn} and \cref{thm:pbmsparseconfsets}.

\begin{definition}
\label{def:cred}
Let $(\Theta_n,\scrG_n)$ with priors $\Pi_n$ be given, denote the
sequence of posteriors by $\Pi(\cdot|\cdot):\scrG_n\times\scrX_n\to[0,1]$.
Let $\scrD_n$ denote a collection of measurable subsets of $\Theta_n$. 
A \emph{sequence of credible sets} $(D_n)$ of
\emph{credible levels} $1-a_n$ (where $0\leq a_n\leq1$, $a_n\downarrow0$)
is a sequence of set-valued maps $D_n:\scrX_n\to\scrD_n$ such that
$\Pi(\Theta_n\setminus D_n(x^n)|X^n=x^n)\leq a_n$.
\end{definition}
Note that the posterior is defined $P^\Pi_n$-almost-surely with respect
to its dependence on the data $X^n$, and consequently, so is any credible
set that is derived from it.
\begin{definition}
\label{def:confcred}
Let $D$ be a (credible) set in $\Theta$ and let
$B=\{B(\theta):\theta\in\Theta\}$ denote a collection of model
subsets such that $\theta\in B(\theta)$ for all $\theta\in\Theta$.
A model subset $C$ is said to be (a confidence set)
associated with $D$ under $B$, if for all $\theta\in\Theta\setminus C$,
$B(\theta)\cap D=\emptyset$.
\end{definition}
The relationship between a credible set $D$ and the model subset $C$
associated with $D$ under $B$ is illustrated in \cref{fig:assoc}
(reproduced from \cite[Figure 1]{Kleijn21}) and
detailed in the following theorem. (The notation $P_n\ctg d_n^{-1}Q_n$
is explained below, see \cref{def:remctg}.)
\begin{figure}[bht]
  \begin{center}
    \parbox{0.75\textwidth}{
      \begin{lpic}{assoc-2(0.30)}
        \lbl[t]{43,133;{\Large $\theta$}}
        \lbl[t]{83,182;{\Large $B(\theta)$}}
        \lbl[t]{150,90;{\Large $D$}}
        \lbl[t]{229,180;{\Large $C$}}
      \end{lpic}
      \caption{\label{fig:assoc}The relation between a
      credible set $D$ and its associated
      confidence set $C$ under $B$
      in Venn diagrams: the extra points $\theta$ in the associated
      confidence set $C$ not included in the credible set
      $D$ are characterized by non-empty intersection
      $B(\theta)\cap D\neq\emptyset$. [Reproduced from
      \cite[Figure 1]{Kleijn21}.]}
    }
  \end{center}
\end{figure}
\begin{theorem}
\label{thm:coverage}
Let $\theta_{0,n}\in\Theta_n$ ($n\geq1$) and $0\leq a_n\leq1$, $b_n>0$
such that $a_n=o(b_n)$ be given. Choose priors $\Pi_n$ and let $D_n$
denote level-$(1-a_n)$ credible sets in $\Theta_n$. Furthermore, for
all $\theta\in\Theta$, let
$B_n=\{B_n(\theta_n)\in\scrG_n:\theta_n\in\Theta_n\}$ and $b_n$
denote sequences such that,
\begin{itemize}
\item[(i.)] prior mass is lower bounded, $\Pi_n(B_n(\theta_0))\geq b_n$,
\item[(ii.)] and for some $d_n\downarrow0$ such that $b_n^{-1}a_n=o(d_n)$,
  $P_{\theta_0,n}\ctg d_n^{-1}\, P_n^{\Pi|B(\theta_0)}$.
\end{itemize}
Then any confidence sets $C_n$ associated with the credible
sets $D_n$ under $B_n$ are asymptotically consistent, \ie, for
all $\theta_0\in\Theta$,
\begin{equation}
  \label{eq:coverage}
  P_{\theta_0,n}\bigl(\,\theta_0\in C_n(X^n)\,\bigr)\to 1.
\end{equation}
\end{theorem}
In most of \cref{sec:pbmuncertainty}, the sets $B_n$ are
simply,
\[
  B_n(\theta_n)=\{\theta_n\},
\]
for every $n\geq1$ and every $\theta_n\in\Theta_n$,
so that the confidence sets $C_n$ associated with
any credible sets $D_n\subset\Theta_n$ under $B_n$ are simply
\emph{equal to} $D_n$. In that case, $P_{\theta_0,n}\ctg c_n^{-1}\,
P_n^{\Pi|B(\theta_0)}$ for any rate $(c_n)$, $c_n\downarrow0$,
so all sequences $a_n=o(b_n)$ are permitted.
Since the prior mass in $B_n(\theta_{0,n})$ is fixed,
\cref{thm:coverage} says that, if we have a
sequence of credible sets $D_n(X^n)\subset\Theta_n$ of high enough
credible levels $1-a_n$, then these $D_n(X^n)$ are also
asymptotically consistent confidence sets (see
\cref{prop:credisconf}). 

With an eye on enlarged credible sets, we note that condition~{\it (ii.)}
of \cref{thm:coverage} says that the sequence $(P_{\theta_0,n})$
is required to be \emph{remotely contiguous} with respect to
$P_n^{\Pi|B(\theta_0)}$ at rate $b_na_n^{-1}$.
\begin{definition}
\label{def:remctg}
Given the spaces $\scrX_n$, $n\geq1$ with two
sequences $(P_n)$ and $(Q_n)$ of probability measures and
a sequence $\rho_n\downarrow0$, we say that
$Q_n$ is $\rho_n$-remotely contiguous with respect to $P_n$,
notation $Q_n\contig \rho_n^{-1}P_n$,
if,
\begin{equation}
  \label{eq:defrc}
  P_n\phi_n(X^n) = o(\rho_n)
  \quad\Rightarrow\quad Q_n\phi_n(X^n)=o(1),
\end{equation}
for every sequence of $\scrB_n$-measurable $\phi_n:\scrX_n\rightarrow[0,1]$.
\end{definition}
According to section~3 in \cite{Kleijn21}, weak relative
compactness of a sequence of re-scaled (inverse) likelihood ratios
is sufficient for remote contiguity.
\begin{lemma}
\label{lem:rcfirstlemmaprime}
Given $(P_n)$, $(Q_n)$, $d_n\downarrow0$,
$(Q_n)$ is $d_n$-remotely contiguous with respect to $(P_n)$
if, under $Q_n$, every subsequence of
$(d_n(dP_n/dQ_n)^{-1})$ has a weakly convergent subsequence.
\end{lemma}
According to Prokhorov's theorem, the condition of
\cref{lem:rcfirstlemmaprime} is equivalent to
uniform tightness: for every $\ep>0$ there exists an $M>0$ such that,
\[
  \sup_{n\geq1} P\Bigl( d_n\Bigl(\frac{dP_n}{dQ_n}\Bigr)^{-1}(X^n)>M
    \Bigr) < \ep.
\]



\bibliography{pbm4.bib}

\end{document}